\documentclass[11pt, reqno]{amsart}

\usepackage{hyperref}
\usepackage{color}

\newtheorem{theorem}{Theorem} 
\newtheorem{lemma}[theorem]{Lemma}
\newtheorem{proposition}[theorem]{Proposition}
\newtheorem{corollary}[theorem]{Corollary}

\theoremstyle{definition}
\newtheorem{remark}[theorem]{Remark}

\usepackage{amssymb}
\renewcommand{\leq}{\leqslant}
\renewcommand{\geq}{\geqslant}
\newcommand{\eps}{\epsilon}

\def\lstar{\op{lg*}}
\newcommand\op[1]{\operatorname{#1}}
\def\P{\mathbf{P}}

\def\calM{\mathcal{M}}
\def\ord{\op{ord}}
\def\type{\op{type}}
\def\Z{Z}

\usepackage{todonotes}

\begin{document}
\title{Permutations With Equal Orders}

\author[Acan]{Huseyin Acan}
            \address{Department of Mathematics, Drexel University, Philadelphia, Pa, 19104}
          \email{ha627@drexel.edu}                 
\author[Burnette]{Charles Burnette\textsuperscript{1}}
           \address{Mathematics Department, Xavier University of Louisiana}
           \thanks{\textsuperscript{1}Some of this work was done while Charles Burnette was 
           a postdoctoral visitor at  Saint Louis University and the Institute of Statistical Science, Academia         Sinica}
            \email{cburnet2@xula.edu}           
\author[Eberhard]{Sean Eberhard}
          \address{Department of Pure Mathematics and Statistics,University of Cambridge,London}
          \email{eberhard.math@gmail.com}           
\author[Schmutz]{Eric Schmutz}
            \address{Department of Mathematics, Drexel University, Philadelphia, Pa, 19104}
             \email{eschmutz@math.drexel.edu}
\author[Thomas]{James Thomas\textsuperscript{2}}
           \thanks{\textsuperscript{2}Portions of this work  appear in James  Thomas' Ph.D. thesis at Drexel University. }
            \address{Department of Mathematics, Drexel University, Philadelphia, Pa, 19104}
            \email{jjt94@dragons.drexel.edu}
            
\begin{abstract}
Let $ \P(\ord \pi = \ord \pi')$ be the probability that two independent, uniformly random  permutations of $[n]$ have the same order.
Answering a question of Thibault Godin, we prove that $\P(\ord \pi = \ord \pi') = n^{-2+o(1)}$ and that $\P(\ord \pi = \ord \pi') \geq \frac12 n^{-2} \lstar n$ for infinitely many $n$.
(Here $\lstar n$ is the height of the tallest tower of twos that is less than or equal to $n$.)

\end{abstract}           
\maketitle
\section{Introduction}
\subsection{The problem}
Let $\pi$ be a random permutation of $[n]$.
Write $\ord \pi$ for the order of a permutation $\pi$, i.e., the least common multiple of its cycle lengths.
The distribution of $\ord \pi$ is an object of basic interest in probabilistic group theory.
For example, a beautiful theorem of Erd\H{o}s and Tur\'an~\cite{MR215908}
asserts that $\log \ord \pi$ is asymptotically normal with mean $\log^2 n / 2$ and variance $\log^3 n / 3$.
Many more features of the distribution of $\ord \pi$ are visible through the lens of the theory of logarithmic combinatorial structures: see for example the 
book of Arratia, Barbour, and Tavar\'e~\cite{MR2032426}.
For example, the largest cycles of $\pi$, which  determine the magnitude of $\ord \pi$ and its divisibility by large primes, follow a Poisson--Dirichlet law.

A more subtle feature of the distribution of $\ord \pi$ is its collision entropy. Recall that the \emph{collision entropy} or \emph{R\'enyi 2-entropy} of a random variable $X$ is defined by
\[
  H_2(X) = - \log \P(X = X'),
\]
where $X'$ is an independent copy of $X$.
In other words, for $X = \ord \pi$, the problem is to estimate
\[
  e^{-H_2(\ord \pi)} = \P(\ord \pi = \ord \pi').
\]
This problem was highlighted recently by 
Godin~\cite{MR3630945}
in connection with automaton groups.
We are grateful to Sergey Dovgal for bringing this problem to our attention.

Let $\type \pi$ denote the cycle type of $\pi$, i.e., the multi-set of its cycle lengths.
Since permutations with the same type have the same order, it is clear that
\[
  \P(\ord \pi = \ord \pi') \geq \P(\type \pi = \type \pi').
\]
Using methods of analytic combinatorics, it was proved by Flajolet, Fusy, Gourdon, Panario, and Pouyanne~\cite[Proposition~4]{MR2274318} 
that
\[
  \P(\type \pi = \type \pi')=\frac{c_{0}}{n^{2}}+O\left(\frac{1}{n^{3}}\right), \hspace{2.5em}
  c_0 = \prod_{k \geq 1} I\left(\frac{1}{k^2}\right) \approx 4.26,
\]
where $I(z) = \sum_{n \geq 0} z^n / n!^2$.
Based on this lower bound and computations, Godin conjectured that
\begin{equation}
  \label{godin-conj}
  \lim_{n \to \infty} n^2 \P(\ord \pi = \ord \pi') = K
\end{equation}
for some constant $K$ with $c_0 \leq K \leq 12$ 
(see \cite[Conjecture~15]{MR3630945}).

It follows from the Erd\H{o}s--Tur\'an limit law for $\log \ord \pi$ that $\P(\ord \pi = \ord \pi') = o(1)$, but establishing any explicit rate of decay is already nontrivial.
A crude bound was established in an earlier version of this paper and in the fifth author's thesis~\cite{Thomas-thesis}. 
Briefly, using estimates for the probability that $\ord \pi$ is coprime to a given integer, one can prove that with high probability there is a prime in the interval $[\log n, 2\log n]$ that divides exactly one of $\ord \pi$ and $\ord \pi'$.
This argument leads to a bound of the form $O(\log \log n / \log n)$,
but does not come close to Godin's conjecture.

We can contrast the effort involved in estimating the collision entropies of $\type \pi$ and $\ord \pi$ even further. Suppose $k = o(n)$ and let $\lambda = \langle1^{\lambda_1}, 2^{\lambda_2}, \ldots\rangle$ be a partition of $k.$ Then
\begin{equation}
\label{largecycle}
\P(\pi\ \text{and}\ \pi'\ \text{have type}\ \langle\lambda, n - k\rangle) = \frac{1}{(n - k)^2}\prod_{j \geq 1} \frac{1}{j^{2\lambda_j}(\lambda_j!)^2} \approx \frac{1}{n^2}\prod_{j \geq 1} \frac{1}{j^{2\lambda_j}(\lambda_j!)^2}.
\end{equation}
Heuristically summing the rightmost approximation over all such $k$ and $\lambda$ yields a substantial partial sum of
\begin{equation}
\label{partialsum}
\frac{1}{n^2}\sum_{k \geq 0} \sum_{\lambda \vdash k} \prod_{j \geq 1} \frac{1}{j^{2\lambda_j}(\lambda_j!)^2} = \frac{c_0}{n^2},
\end{equation}
where $\lambda \vdash k$ has the usual meaning that $\lambda$ is a partition of $k$ (i.e. a multiset of positive integers whose sum is $k$).
This together with the analysis of~\cite{MR2274318} shows that the main contribution to $n^2\P(\type \pi = \type \pi')$ comes from pairs of permutations having a cycle of length $n - o(n).$ Motivated by this, one may ask whether at least $n^2\P(\ord \pi = \ord \pi' \wedge E )$ is bounded, where $E$ is the event that 
 $\pi$ and $\pi^{'}$ each have a cycle of length at least $n - k(n),$ where  $k(n)$ is some slowly growing function. This, however, is not the case.

In this paper we prove two main results.
First, we refute \eqref{godin-conj} by showing that
\[
  \limsup_{n \to \infty} n^2 \P(\ord \pi = \ord \pi') = \infty.
\]
Quantitatively, we show that there is a sequence $n_i \to \infty$ such that if $\pi, \pi'$ are drawn independently from $S_{n_i}$ then
\[
  \P(\ord \pi = \ord \pi') \geq \P(\ord \pi = \ord \pi = n_i - o(n_i)) \geq \frac12 n_i^{-2} \lstar n_i,
\]
where $\lstar n$ is the height of the tallest tower of twos that does not exceed $n$. This precludes any heuristic similar to \eqref{largecycle} and \eqref{partialsum} from succeeding here. On the other hand we show that \eqref{godin-conj} is nearly true, in the sense that
\begin{equation}
  \label{main-upper-bound}
  \P(\ord \pi = \ord \pi') \leq n^{-2+o(1)}.
\end{equation}
It would be interesting to estimate $\P(\ord \pi = \ord \pi')$ more precisely, but this appears to be a complicated question tied to arithmetic considerations about $n$.

For a broader perspective, readers may be interested in the survey of Niemeyer, Praeger, and Seress on the applications of probabilistic and enumerative techniques to the analysis of group-theoretic algorithms \cite{MR3026186}.

\subsection{Analytic Combinatorics}
Analytic combinatorics relates the analytic behaviour of a generating function to the asymptotic behaviour of its coefficients.
While the problem of estimating $\P(\type \pi = \type \pi')$ is well-suited to the methods of analytic combinatorics, the same does not seem to be true of $\P(\ord \pi = \ord \pi')$.
We offer some brief comments about why this may be.

Elementary combinatorial techniques are sufficient for enumerating the ordered pairs of conjugate permutations.
As a result, the numbers $\P(\type \pi = \type \pi')$ are expressible as the coefficients of a well-behaved infinite product generating function closely related to the cycle index of the symmetric group (as explained in ~\cite[Section~4.2]{MR2274318}.

In contrast consider $\P(\ord \pi = \ord \pi')$.
For any fixed positive integer $m$, the exponential formula yields
\begin{equation}
\label{divides}
F_{m}(x)
= \sum\limits_{n}\P(\ord \pi ~\text{divides}~ m) x^{n}
= \exp\left(\sum\limits_{d|m}\frac{x^{d}}{d}\right).
\end{equation}
An application of M{\"o}bius inversion to \eqref{divides} thus yields
\begin{equation}
 \label{egf-ordm}
G_{m}(x)
= \sum\limits_{n} \P(\ord \pi = m) x^n
= \sum\limits_{d|m}\mu\left(\frac{m}{d}\right)F_{d}(x).
\end{equation}
Using M{\"o}bius and Lagrange inverson, and the saddle-point method, 
Wilf~\cite{MR854561}
used \eqref{egf-ordm} to derive an asymptotic formula for $\P(\ord \pi = m)$ for \emph{fixed} $m$,
but as $m$ grows Wilf's formula becomes more complicated
and the asymptotics are less well-understood.
In the special case of $m = n$ there is a theorem of 
Warlimont~\cite{MR498814}
 that
\[
  \P(\ord \pi = n) = 1/n + O(1/n^2),
\]
and this estimate has been extended by
Niemeyer and Praeger~\cite{MR2335705}
to various other values of $m$.
A general understanding of $\P(\ord \pi = m)$ is lacking, and indeed complicated for arithmetic reasons.
As such, one cannot simply plug these asymptotic estimates into the sum $\sum_{m}\P(\ord \pi =m)^{2}$ to answer Godin's question.


There is a rich literature about methods for extracting the coefficients of
multivariate generating functions \cite{MR3088495,MR2916463}.
Certainly we may define a bivariate generating function
$H(x,y)=\sum_{m}G_{m}(x)G_{m}(y)$, and
\begin{equation}
\label{answer}
\P(\ord \pi=\ord \pi^{'})=[[x^{n}y^n]] H\left(x,y\right).
\end{equation}
In some formal sense this is an answer, but we do not see any way to extract an asymptotic
formula from \eqref{answer}.

Analytic combinatorics, by itself, is likely inadequate for attaining a thorough asymptotic analysis of the sequence $\P(\ord\pi = m)$ because
the order of a permutation depends on arithmetic data not easily extracted from the classical generating functions associated with permutations.
Any hope for a purely symbolic calculus that can handle the sequence $\P(\ord \pi = \ord \pi')$ might hinge on techniques that are more in the realm of analytic number theory,
such as a Mellin transform or a Dirichlet series generating function.

Another notable obstruction to a generating-function-based approach is the apparently erratic dependence of $\P(\ord \pi = \ord \pi')$ on $n$,
which may be observed numerically.
If the sequence $\P(\ord \pi = \ord \pi')$ were realized as the coefficients of a generating function, the behaviour of that function near its singularity would have to be similarly complicated.

\subsection{The anatomy of integers}

In sharp relief from the beautiful formalism of analytic combinatorics,
our proof of \eqref{main-upper-bound} is dirty and hands-on,
and more closely connected with the ``anatomy of integers'':
see Granville~\cite{granville} for an explanation of this term,
and Ford~\cite{ford} or the book of 
Hall and Tenenbaum~\cite{MR964687}
for a sense of the scope of the theory.
We have mentioned already that $\log \ord \pi$ is asymptotically normal with mean $\log^2 n / 2$ and variance $\log^3 n / 3$, and that the largest cycles of $\pi$ are distributed asymptotically according to the Poisson--Dirichlet law.
By further analyzing the distribution of the cycles of $\pi$ we show that apart from an exceptional event of probability $n^{-1+o(1)}$, including for instance the event that $\pi$ is an $n$-cycle or an $(n-1)$-cycle,
the integer $m = \ord \pi$ will have many large prime divisors,
so many in fact that the collision probability $\P(\ord \pi' = m)$ is negligible.
It follows that the probability that $\ord \pi = \ord \pi'$ is dominated by the event that $\pi$ and $\pi'$ are both exceptional.
\section{Disproof of Godin's Conjecture}

\def\Tow{\operatorname{Tow}}

 The results in this section are based on the third author's  {\tt math}{\it overflow} post~\cite{Godin-16}. Define $\Tow(h)$ to  be a tower of twos of height $h$, i.e., $\Tow(0)=1$,
 and for $h>0$, $\Tow(h)=2^{\Tow(h-1)}.$ Also define
 $\lstar n=\max\lbrace h: \Tow(h)\leq  n\rbrace.$

 \begin{theorem}\label{thm:lower bound} For infinitely many  positive integers $n$,
 \[
 \P(\ord \pi =\ord \pi')
 \geq \frac{\lstar n}{2n^2}.
 \]
\end{theorem}
\begin{proof}
For a positive integer $n$, let $K_n =\{k:1\leq k < n/2\ \text{and}\ k! \text{ divides } n-k\}$. If $\pi$ has a cycle of length $n-k$, with $k \in K_n$, then all other cycles have length at most $k$. Since the lengths of these other cycles are at most $k$, they all divide $k!$, which in turn divides $n-k$ (by the definition of  $K_n $). Therefore $\ord \pi = n-k$. 
The probability that $\pi$ has a cycle of length $n-k$ is exactly $1/(n-k)$. Since $n-k>\frac{n}{2}$,  these events are disjoint, since there cannot be more than one cycle of length greater than $\frac{n}{2}$.  We therefore have
\[
 \P(\ord \pi = \ord \pi')
 \ge \sum_{k \in K_n} \frac{1}{(n-k)^2}
 \ge \frac{|K_n|}{n^2}.
\]
Now consider the subsequence $(n_i)_{i\geq 1}$ defined by $n_1=3$ and $n_{i+1}=n_i+n_i!$ for $i\ge 1$. We will  prove that  the sets $K_{n_{i}}$ are nested and that  $|K_{n_{i}}| = i$ for all $i$.
From the definition of $K_{n}$, it is easy to check that
 \begin{itemize}
 \item  $K_{n_{1}}=\lbrace 1\rbrace$;
  \item $n_{i}\not\in K_{n_{i}}$;
  \item $k \in K_{n_i} \implies k \in K_{n_{i+1}}$, since if $k! \mid n_i - k$ and $k \leq n_i$ then also $k! \mid n_{i+1} - k$;
 \item $n_{i}\in K_{n_{i+1}}$, for the same reason.
  \end{itemize}
  Also note that
   \begin{itemize}
  \item  $k\not\in  K_{n_{i+1}}$ for $k>n_i$ since  $k!$ is too big to divide $n_{i+1}-k$;
 \item
 if  $k<n_i$ and $k\in K_{n_{i+1}}$, then we already have $k \in K_{n_i}$,
 since $k<n_i \implies  k! \mid n_i!$,  which in turn implies $k! \mid n_i -k$.
 \end{itemize}
We therefore have
  $K_{n_{i+1}} = K_{n_i} \cup \{n_i\}$ and $n_i\not\in   K_{n_i}$,
  so inductively
  \[
    K_{n_i} = \{1, n_1, n_2, \dots, n_{i-1}\}.
  \]
 This proves that $|K_{n_i}|=i$ for all $i$.
Since  $|K_{n_i}|\rightarrow \infty$ as $i
 \rightarrow\infty$, it is now clear that  $\limsup\limits_{n\rightarrow\infty} n^{2} \P(\ord \pi = \ord \pi') =\infty.$ 

To finish proving Theorem \ref{thm:lower bound}, we need to find a lower bound for $i$ that is expressed in terms of 
 $n_{i}$.  
 Since $2^{n^2} \ge (n+1)! \ge n!+n$ for any positive integer $n$, we have
\begin{align*}
\lstar(n_{i+1}) = \lstar(n_i!+n_i) &\le \lstar(2^{n_i^2}) \\
&= 1+ \lstar(n_i^2) \\
& \le 1+ \lstar(2^{n_i})  = 2+\lstar(n_i).
\end{align*}
It follows from induction on $i$ that  $\lstar(n_i) \le 2i$ or equivalently, $i \ge \lstar(n_i)/2$. Hence
\[
\frac{|K_{n_i}|}{n_i^2} = \frac{i}{n_i^2} \ge \frac{\lstar(n_i)}{2n_i^2}.\qedhere
\]
\end{proof}

\section{Main Proposition and Proof Sketch}
\label{sec:main-prop}

Throughout let $\pi$ be a random permutation of $[n]$.
Our main result is the following.
\newtheorem*{mainthm}{Theorem 3.1}
\newtheorem*{maincor}{Corollary 3.2}
\begin{mainthm}
\label{main-lemma}
There is a set $\calM$ with the following properties.
\begin{enumerate}
  \item If $m \notin \calM$ then $\P(\ord \pi = m) = O(n^{-100})$.
  \item $\P(\ord \pi \in \calM) \leq n^{-1+o(1)}$.
\end{enumerate}
\end{mainthm}

Although the proof of Theorem~\ref{main-lemma} is postponed, we can
immediately deduce a non-trivial upper bound for the probability that
two random permutations have the same order.
\begin{maincor}
\label{corollary}
  $\P(\ord \pi = \ord \pi')  \leq n^{-2+o(1)}.$
 \end{maincor}
\begin{proof}
By considering whether the collision occurs in $\calM$ or $\calM^c$ we have
\begin{align*}
  \P(\ord \pi = \ord \pi')
  &\leq \P(\ord \pi \in \calM)^2 + \sum_{m \notin \calM} \P(\ord \pi = m)^2 \\
  &\leq \P(\ord \pi \in \calM)^2 + \max_{m\notin \calM} \P(\ord \pi = m).
\end{align*}
The first term is bounded by $n^{-2+o(1)}$ and the second term is bounded by $O(n^{-100})$.
\end{proof}

For the proof of Theorem \ref{main-lemma}, we construct a specific  example of such a set $\calM$. For the remainder of this paper, let  $\delta = \delta(n) = 1/\log\log\log n$, and let $\eta = e^{-10/\delta}=\frac{1}{(\log\log n)^{10}}$, though the specific choice is largely irrelevant: all we require is that $\delta$ and $\eta$ decay sufficiently slowly, with $\delta$ decaying much more slowly than $\eta$. Let $\calM$ be the set of all positive integers $m$ having at most $\delta \log n$ distinct prime divisors $p > n^\eta$.

Let us now informally sketch the proof of Theorem~\ref{main-lemma} (some readers may prefer to skip ahead to the next section for  the rigorous proofs). It suffices to consider the case where $\pi$ has $k \geq 2 \delta \log n$ cycles, because all except $n^{-1+o(1)}$ permutations have this property (recall $\delta = o(1)$). These $k$ cycles will be drawn at random from $\{1, \dots, n\}$ according to a harmonic weighting (conditional on their sum being $n$). Using Mertens' third theorem to bound the harmonic weight of the set of $n^\eta$-smooth numbers, we expect at least half of our $2\delta \log n$ cycles to fail to be $n^\eta$-smooth. Therefore we expect $\ord \pi$ to be divisible by some $\delta \log n$ primes $p > n^\eta$, proving part (2) of Theorem~\ref{main-lemma}. The proof of part (1) is easier, and follows from a simple union bound over all the ways that the cycles of $\pi$ might be divisible by the primes dividing $m$.

\section{Proof of Theorem~\ref{main-lemma}, part (2)}

Write $\Z = \Z(\pi)$ for the number of cycles in a random $\pi\in S_n$. It is well known that $\Z -1$ is approximately Poisson with parameter $\log n$. (See, for example, the final section of \cite{MR969872} for tail bounds.) The following lemma's quantitative  formulation is  particularly convenient for us.

\begin{lemma}
Let $n, k \geq 1$, and let $\pi \in S_n$ be random. Then
\[
  \P(\Z(\pi) = k) \leq \frac{1}{n} \frac{h_n^{k-1}}{(k-1)!},
\]
where $h_n = \sum_{j=1}^n 1/j$.
\end{lemma}
\begin{proof}
Write $p_{n,k}$ for $\P(\Z(\pi) = k)$. From Cauchy's formula for the number of permutations in a conjugacy class, we have
\[
  p_{n,k} = \sum \frac{1}{c_1! \cdots c_n! 1^{c_1} \cdots n^{c_n}},
\]
where the sum ranges over all $c_1, \dots, c_n \geq0$ such that $\sum_{i=1}^n c_i = k$ and $\sum_{i=1}^n i c_i = n$. We can ``smooth this out'' by using
\[
  p_{n,k} = \frac{1}{n} \sum_{j=1}^n p_{n-j, k-1}
\]
which follows from conditioning on the length of one of the cycles of $\pi$. Thus we have
\begin{align*}
  p_{n,k}
  &= \frac{1}{n} \sum_{j=1}^n \sum_{\substack{\sum c_i = k-1 \\ \sum i c_i = n-j}} \frac{1}{c_1! \cdots c_n! 1^{c_1} \cdots n^{c_n}} \\
  &\leq \frac{1}{n} \sum_{\sum c_i = k-1} \frac{1}{c_1! \cdots c_n! 1^{c_1} \cdots n^{c_n}} \\
  &= \frac{1}{n} \frac{h_n^{k-1}}{(k-1)!}.
\end{align*}
The last line is an application of the multinomial theorem.
\end{proof}

Using  Stirling's formula, and  monotonicity of the bound  $ \frac{1}{n} \frac{h_n^{k-1}}{(k-1)!}$ as a function of $k$,
we can prove the following corollary.
\begin{corollary} 
\label{tails}
The probability that $\pi$ has $o(\log n)$ cycles is $n^{-1+o(1)}$, and the probability that $\pi$ has more than $10 \log n$ cycles is $O(n^{-14})$.
\end{corollary}
\begin{proof}
Let $\xi =\frac{h_{n}}{\omega}$, where ${\omega}=\omega(n)\rightarrow \infty.$
By calculating the ratios of successive terms, one can verify that the bound $\frac{1}{n} \frac{h_n^{k-1}}{(k-1)!}$ is increasing as a function of $k$ when $k\leq \xi +1$. Thus
 \[ \P(\Z \leq \xi+1)\leq (\xi+1)\frac{1}{n}\frac{h_{n}^{\xi}}{(\xi+1)!}\leq  \frac{1}{n}(e \omega)^{\xi}=n^{-1+o(1)}.\]
 Similarly, when $k>10h_{n}$, the bound is decreasing as a function of $k$.  In this range, a 
 crude version of Stirling's
 formula yields  
 \[ 
 \frac{h_n^{k-1}}{k!}\leq \left(\frac{eh_{n}}{k}\right)^{k} \leq \left(\frac{e}{10}\right)^k.
 \]
 Therefore
\[ 
\P(\Z \geq 10\log n) \leq  \frac{1}{n} \sum\limits_{k\geq  10h_{n}} \left(\frac{e}{10}\right)^{k}=O\left(n^{10\log(e/10)-1}\right).		\qedhere
\] 
 \end{proof}
We use only Corollary \ref{tails} in the proof,  but a similar argument establishes that, for fixed positive $\eps$, 
 the probability that $\pi$ has more than $(1+\eps) \log n$ cycles is bounded by $n^{-f(\eps) + o(1)}$, where
$ f(\eps) = (1 + \eps) \log ( 1+ \eps) - \eps.$

\begin{lemma}\label{A-lemma}
	Let $A_1, \dots, A_\Z$ be the cycle lengths of $\pi$ in a random order. Then for any $k\geq 0$ and any $k$-tuple $(a_1, \dots, a_k)$ of positive integers such that $a_1 + \cdots + a_k = n$ we have
	\[
	  \P(\Z = k, A_1=a_1, \dots, A_k=a_k) = \frac{1}{k!} \frac{1}{a_1 \dots a_k}.
	\]
\end{lemma}
\begin{proof}
	Let the multiplicities among $a_1, \dots, a_k$ be $m_1, \dots, m_s$ (so that $\sum_i m_i = k$). Then by Cauchy's formula the probability that this cycle type arises is
	\[
		\frac{1}{m_1! \cdots m_s! a_1 \cdots a_k}.
	\]
	When these cycles are ordered randomly, the probability that we get $a_1, \dots, a_k$ in order is
	\[
	  \binom{k}{m_1\,\cdots\,m_k}^{-1}.
	\]
	The result follows from multiplying the previous two displays.
\end{proof}

The combined message of the previous two lemmas is that we may assume $\pi$ has between $\delta \log n$ and $10 \log n$ cycles (for any slowly decaying $\delta$), while, conditional on $k$, these cycles are distributed roughly independently according to a harmonic weighting.

For any set $S$ of integers,  let us call $h_S = \sum_{j \in S} 1/j$ the \emph{harmonic weight} of $S$.
If $P$ is any set of prime numbers,  a  positive integer  $n$ is $P$-smooth iff all prime divisors of
$n$ are elements of $P$.

\begin{lemma}\label{harmonic-weight}
	Let $N\geq 1$. Let $P$ be the set of all primes $p \leq N$, as well as some $o(N)$ further primes. Then the harmonic weight of the set of $P$-smooth numbers is
	\[
	  (1 + o(1)) e^\gamma \log N,
	\]
{where $\gamma$ is the Euler--Mascheroni constant.}
\end{lemma}

\begin{proof}
	The harmonic weight of the set of $P$-smooth numbers is
	\[
	  \prod_{p \in P} (1-1/p)^{-1}.
	\]
	By Mertens' third theorem we have
	\[
	  \prod_{p \leq N} (1-1/p)^{-1} \sim e^\gamma \log N.
	\]
	On the other hand we have
	\[
	  \prod_{p \in P, p > N} (1-1/p)^{-1}
	  = \exp \sum_{p \in P, p > N} O(1/p) = e^{o(1)}.\qedhere
	\]
\end{proof}

 Recall that  $\delta  = 1/\log\log\log n$, and $\eta = e^{-10/\delta}=\frac{1}{(\log\log n)^{10}}$. With this choice of
 $\delta$ and $\eta$ we have the following proposition.
\begin{proposition}
	Let $\pi$ be drawn from $S_n$ uniformly at random. Then apart from an event of probability $n^{-1+o(1)}$, $\pi$ has at least $\delta \log n$ cycles and $\ord \pi$ is divisible by at least $\delta \log n$ primes $p > n^\eta$. 
\end{proposition}

\begin{proof}
	Let $A_1, \dots, A_\Z$ be the cycle lengths of $\pi$ in a random order. By Lemma~\ref{A-lemma}, provided that $a_1 + \cdots + a_k = n$ we have
	\[
	  \P(\Z = k, A_1=a_1, \dots, A_k=a_k) = \frac{1}{k!} \frac{1}{a_1 \cdots a_k}.
	\]
	
	Define sets of primes $P_i$ as follows:
	\begin{enumerate}
		\item Let $P_0$ be the set of all primes $p \leq n^\eta$.
		\item For $0 < i \leq k$, if $A_i$ is $P_{i-1}$-smooth, put $P_i = P_{i-1}$. Otherwise pick a prime $p_i \notin P_{i-1}$ dividing $A_i$ (the smallest such, say), and let $P_i = P_{i-1} \cup \{p_i\}$.
	\end{enumerate}
	Each set $P_i$ contains at most $k$ primes $p > n^\eta$, so as long as $k = o(n^\eta)$ Lemma~\ref{harmonic-weight} implies that the set of $P_i$-smooth numbers has harmonic weight at most $2 \eta h_n$. 
	
	Let $I$ be the set of indices $i\in\{1, \dots, k\}$ such that $A_i$ is $P_{i-1}$-smooth (and hence $P_i=P_{i-1}$). Assuming $k \geq 2\delta \log n$, if $|I| \leq k/2$ then we find that $P_k$ contains at least $\delta \log n$ distinct primes $p > n^\eta$, as desired. We will bound the probability that $|I| > k/2$.
	
	Let $E_k$ be the event that $\pi$ has $k$ cycles and $|I| > k/2$. Then, assuming $2\delta \log n \leq k \leq 10 \log n$,
	\begin{align*}
		\P(E_k)
		&= \sum_{I_0: |I_0| > k/2} \P(\text{$\Z(\pi)=k$ and $I=I_0$}) \\
		&= \sum_{I_0: |I_0| > k/2} \sum_{\substack{a_1, \dots, a_k \geq 1 \\ a_1 + \cdots + a_k = n}} \frac{1}{k!} \frac{1_\text{$a_i$ is $P_{i-1}$-smooth for each $i \in I_0$}}{a_1 \cdots a_k} \\
		&\leq \sum_{I_0: |I_0| > k/2} \frac{1}{k!} h_n^{k - |I_0|} (2 \eta h_n)^{|I_0|} \\
		& \leq \frac{h_n^k}{k!} 2^k (2\eta)^{k/2} \\
		& \leq \frac{h_n^k}{k!} (8\eta)^{\delta \log n} \\
		& \leq \frac{h_n^k}{k!} n^{-10 + o(1)}.
	\end{align*}
	 Hence
	\[
	  \P\left( \bigcup_{2\delta \log n \leq k \leq 10 \log n} E_k\right) \leq   e^{h_{n}} n^{-10 + o(1)}=    n^{-9+o(1)}.
	\]
	On the other hand, by Corollary~\ref{tails} the probability that $\pi$ has either fewer than $2\delta \log n$ cycles or more than $10 \log n$ cycles is bounded by $n^{-1+o(1)}$.
	 This proves the lemma.
\end{proof}

This finishes the proof of part (2) of Theorem~\ref{main-lemma}.

\section{Proof of Theorem~\ref{main-lemma}, part (1)}

Recall that  $\delta  = 1/\log\log\log n$, and $\eta = e^{-10/\delta}=\frac{1}{(\log\log n)^{10}}$.

\begin{lemma}
	Let $m$ be an integer having at least $\delta \log n$ prime divisors $p > n^\eta$. Then
	\[
		\P(\ord \pi = m) \leq e^{-c \delta \eta \log^2 n}.
	\]
\end{lemma}
\begin{proof}
	Recall that the cycle lengths of a random permutation can be sampled using the following process. Start by picking $a_1$ uniformly from $\{1, \dots, n\}$. If $a_1 < n$, pick $a_2$ uniformly from $\{1, \dots, n - a_1\}$, etc. The process continues until $a_1 + \cdots + a_k = n$.
	
	Fix a set $P$ of $\lceil \delta \log n \rceil$ prime divisors $p > n^\eta$ of $m$. Now sample $\pi \in S_n$ using the process just described. For each fixed $i$ and $p$, the probability that $a_i$ is divisible by $p$ is at most $1/p$, independently of the previous steps in the process. In fact, for any set of primes $p_1, \dots, p_t \in P$, the probability that $a_i$ is divisible by each of $p_1, \dots, p_t$ is at most $1/(p_1 \cdots p_t)$. On the other hand, in order that $\ord \pi = m$, for each $p \in P$ there must be an index $i$ such that $a_i$ is divisible by $P$.
	
	The event that $\pi$ has more than $(\log n)^3$ cycles is negligible (it has probability $o(e^{-c (\log n)^3})$). On the other hand, the probability that $\pi$ has at most $(\log n)^3$ cycles and that for each $p \in P$ there is some $i$ such that $a_i$ is divisible by $p$ is bounded by
	\begin{align*}
	  ((\log n)^3)^{|P|} \cdot \prod_{p \in P} 1/p
	  &\leq (\log n)^{O(\log n)} (n^{-\eta})^{\delta \log n} \\
	  &\leq e^{-c \delta \eta \log^2 n}. \qedhere
	\end{align*}
\end{proof}

This finishes the proof of Theorem~\ref{main-lemma}.

\section{Conclusion}
While we have established that $\P(\ord \pi = \ord \pi')$ is generically larger than $O(1/n^2)$ but no larger than $n^{-2+o(1)}$,
its exact order of magnitude remains mysterious and appears to be linked with arithmetical properties of $n$, as in the proof of Theorem~\ref{thm:lower bound}.
Establishing more precise estimates should be of interest to anyone who considers
themself to be a problem solver (in the sense of
Gowers's essay \cite{MR1754768}), just because it is an easily stated problem that is not readily solved.
We list here a few related observations and open questions.
\begin{enumerate}

 \item What is the $\liminf$ of $n^2 \P(\ord \pi = \ord \pi')$ as $n$ tends to infinity? The integers $n$ constructed by Theorem~\ref{thm:lower bound} have a particular arithmetic form. What is the behaviour for $n$ of the form $k!+1$?

 \item What is $\max_{m}\P(\ord\pi =m)$, and for what value(s) of $m$ is it attained?
 	Theorem~\ref{main-lemma} gives an upper bound of $n^{-1+o(1)}$ for this probability.
  Clearly the max is at least $1/n$, since $\pi$ is an $n$-cycle with probability $1/n$.
  In fact the max is at least $1/(n-1)$, for the same reason but with $(n-1)$-cycles.
  The answer may be close to this,
  but we saw in the proof of Theorem~\ref{thm:lower bound} that
  \[
    \P(\ord \pi = n-k) \ge 1/(n-k)
  \]
  for any $k \in K_n$, so the maximum can be larger.
  This problem was mentioned by Erd\H{o}s and Tur\'an in 
  \cite{MR232833}.
  \item Let $\pi_n$ be a random element of $S_n$.
 The quantity $\P(\ord \pi_n = m)$ as a function of $m$ and $n$ can be very sensitive to the value of $n$. For example, if $n$ is prime then $\P(\ord \pi_n = n) = 1/n$ but $\P(\ord \pi_{n-1} = n) = 0$.
         \item As a generalization of Godin's problem, one might consider symmetric groups of different sizes.
Consider random permutations $(\pi_1,\pi_2)\in S_{n_{1}}\times S_{n_{2}}$, and estimate the probability they have the same order. An upper bound is immediate from Corollary~\ref{corollary} and the Cauchy--Schwarz inequality:
\begin{align*}
  \P(\ord \pi_1 = \ord \pi_2)
  &= \sum_m \P(\ord \pi_1 = m) \P(\ord \pi_2 = m) \\
  &\leq \P(\ord \pi_1 = \ord\pi'_1)^{1/2} \P(\ord \pi_2 = \ord\pi'_2)^{1/2} \\
  &\leq n_1^{-1+o(1)} n_2^{-1+o(1)}.
\end{align*}
 \end{enumerate}

\bibliographystyle{alpha}
\bibliography{ABEST-arxiv-v3}

\end{document}